\documentclass{lms_notitle}
\usepackage{amscd}
\usepackage{amsmath}
\usepackage{amssymb}
\usepackage{graphicx}

\newtheorem{theorem}{Theorem}[section] 
\newtheorem{lemma}[theorem]{Lemma}     
\newtheorem{corollary}[theorem]{Corollary}
\newtheorem{proposition}[theorem]{Proposition}

\newnumbered{assertion}{Assertion}    
\newnumbered{conjecture}{Conjecture}  
\newnumbered{definition}{Definition}
\newnumbered{hypothesis}{Hypothesis}
\newnumbered{remark}{Remark}
\newnumbered{note}{Note}
\newnumbered{observation}{Observation}
\newnumbered{problem}{Problem}
\newnumbered{question}{Question}
\newnumbered{algorithm}{Algorithm}
\newnumbered{example}{Example}
\newunnumbered{notation}{Notation} 

\newcommand{\be}{\begin{eqnarray}}
\newcommand{\ee}{\end{eqnarray}}
\newcommand{\benn}{\begin{eqnarray*}}
\newcommand{\eenn}{\end{eqnarray*}}
\newcommand{\bse}{\begin{equation}}
\newcommand{\ese}{\end{equation}}
\newcommand{\bsenn}{\begin{displaymath}}
\newcommand{\esenn}{\end{displaymath}}

\newcommand{\R}{\mathbb{R}}
\newcommand{\logand}{\;\;{\rm and }\;\;}
\newcommand{\where}{\;\;{\rm where }\;\;}

\title[Diffusion and consensus in digraphs]
 {Diffusion and consensus on weakly connected directed graphs} 

\author{J. J. P. Veerman, E. Kummel}

\classno{05C40, 05C50, 05C20, 05C81, 94C15.}

\extraline{Maseeh Dept. of Math. and Stat., Portland State Univ., Portland, OR, USA}

\begin{document}
\maketitle

\begin{abstract}
\noindent Let $G$ be a weakly connected directed graph with asymmetric graph Laplacian ${\cal L}$.
Consensus and diffusion are dual dynamical processes defined on $G$ by $\dot x=-{\cal L}x$ for consensus
and $\dot p=-p{\cal L}$ for diffusion.  We consider both these processes as well their discrete time analogues.  We define a basis of row vectors $\{\bar \gamma_i\}_{i=1}^k$ of the
left null-space of ${\cal L}$ and a basis of column vectors $\{\gamma_i\}_{i=1}^k$
of the right null-space of ${\cal L}$ in terms of the partition
of $G$ into strongly connected components.
This allows for complete characterization of the asymptotic behavior
of both diffusion and consensus --- discrete and continuous --- in terms of these eigenvectors.

As an application of these ideas, we present a treatment of the pagerank algorithm
that is dual to the usual one. We further show that the \emph{teleporting}
(see below) feature usually included in the algorithm is not strictly necessary.

Together with \cite{CV}, this is a complete and self-contained treatment of
the asymptotics of consensus and diffusion on digraphs.
Many of the ideas presented here can be found scattered in the literature, though
mostly outside mainstream mathematics and not always with complete proofs.
This paper seeks to remedy this by providing a compact and accessible survey.
\end{abstract}


\section{Introduction} 
\label{chap:intro}
\noindent Directed graphs are an important generalization of undirected graphs because they have wide-ranging
applications.  Examples include models of the internet \cite{Brod} and social networks \cite{Carr},
food webs \cite{May}, epidemics \cite{Jom}, chemical reaction networks \cite{Rao},
databases \cite{Ang}, communication networks \cite{Ahls}, and
networks of autonomous agents in control theory \cite{Fax}; to name but a few.
In each of these examples, the relation defining adjacency between neighboring vertices is often not symmetric.
Such networks are most naturally
represented by directed graphs. Still, in many standard references, directed graphs rate
no more than one section or chapter out of many.
A notable exception to this is \cite{bang}, but even here,
little mention is made of the \emph{algebraic} theory for digraphs.
One reason for this is perhaps that adjacency matrices (and hence Laplacians) of undirected graphs
are symmetric and thus there is a complete basis of orthonormal eigenvectors, which is not
generally true for digraphs.

One aim of this article is to collect and explain some important results in algebraic digraph theory
in a unified context. We will see that the lack of symmetry, referred to above,
can actually be used to reveal a great deal of structure in general digraphs that has
\emph{no counterpart} in undirected graphs.
Many of these results are known or published in varied disciplines outside the mainstream
of mathematics, but frequently are only stated in the case of strongly connected graphs and
not always with proofs that apply more generally.

In considering applications, such as those mentioned in the first paragraph,
we found it a useful heuristic to fix once and for all \emph{the direction of the edges in the digraph
as the direction of the flow of information}. In most cases, this is a convenient way to determine
the (ultimately arbitrary) direction of the arrows of the graph, and unify the treatment of these
different applications. In our article, the direction of the flow
of information is exemplified by the consensus problem (see below): if $i$'s opinion is influenced by
$j$, then there is an edge from $j$ to $i$. We will see that a random walk on the same digraph
naturally is defined to go in the \emph{opposite} direction, i.e. from $i$ to $j$.
One of our main conclusions
is that these processes in opposite directions can be considered as \emph{duals} of one another.

Finally, we apply some of these insights to the pagerank algorithm most famously associated with Google founders Brin and Page.
The usual treatment of this algorithm involves a random walk in a graph defined by web links, together
with some random jumps. We give a dual interpretation in terms of how information
added to one page influences the collective network as a whole.

The outline of this article is as follows. In the next section, we give important definitions
and some background relating to our topic. In Section \ref{chap:prior},
we describe the structure of the right kernel of the random walk Laplacian ${\cal L}$ of a directed graph $G$
(with or without teleporting). In particular, we show how it relates to the way strongly connected components of $G$ are linked. Most of this section is based on \cite{CV}.
In Section \ref{chap:random-walks}, we give the structure of the left kernel of ${\cal L}$ and similarly relate
it to how the different components of $G$ are linked. Furthermore, we show that the left kernel
encodes the asymptotic behavior of a random walk defined on $G$. In Section \ref{chap:left-and-right},
we use both left and right kernels, and show that the asymptotic behavior of the random walk on the one hand
and consensus on the other, are natural duals. In Section \ref{chap:ranking}, we give the dual
interpretation of the classical pagerank algorithm, and show that the common practice
of adding ``teleporting" to so-called ``dangling nodes" adds no new information
to the ranking order. Finally, in Section \ref{chap:examples}, we illustrate our definitions
and algorithms in one (small) graph.

\section{Definitions and Background}
\label{chap:definitions}
\noindent We now give an overview of the most relevant definitions concerning digraphs.

Let $G$ be a directed graph on $V = \{1,\dots,n\}$ with (directed) edge-set $E\subseteq V\times V$.
A directed edge $j\rightarrow i$ will be referred to as $ji$.
We also write $j\rightsquigarrow i$ if there exists a directed path in $G$ from vertex $j$ to vertex
$i$. To each (directed) edge $ji\in E$, we assign a positive weight $w_{ji}>0$. The \emph{adjacency matrix}
$Q$ is the $n\times n$ matrix whose $ij$ entry equals $w_{ji}>0$ if $ji\in E$ and zero otherwise.

The notion of ``connectedness" in undirected graphs has several natural extensions to directed graphs.
See \cite{bang} for more information. A digraph $G$ with vertex set $V$ is \emph{strongly connected} if
for all $i$ and $j$ in $V$, there is a path $i\rightsquigarrow j$ (and thus also a path
$j\rightsquigarrow i$). Such components are called \emph{strong components}.
A much weaker form of connectedness is the following.
$G$ is \emph{weakly connected} if for all $i$ and $j$ in $V$, there is an undirected path in G from $i$ to $j$.
Equivalently, $G$ is weakly connected if its underlying undirected graph (all edges replaced by undirected edges)
is connected. There is an intermediate form of connectedness, called \emph{unilateral connectedness} in
\cite{bang}. $G$ is \emph{unilaterally connected} if for all $i$ and $j$ in $V$, there is a
undirected path $i\rightsquigarrow j$ or $j\rightsquigarrow i$. A digraph that is not weakly connected
is disconnected. To study its properties, it is sufficient to study the properties of its weakly
connected components. Thus in this article, we assume without loss of generality that our digraphs
are weakly connected, but not necessarily unilaterally connected. From now on, we will
abbreviate \emph{(weighted) directed graph} to \emph{graph} unless misunderstanding is possible.

\begin{definition} \cite{CV}
For any vertex $j$, we define the \emph{reachable set} $R(j)$ of vertex $j$ to be the set containing
$j$ and all vertices $i$ such that $j \rightsquigarrow i$.
\label{def:reachable}
\end{definition}

\begin{definition} \cite{CV}
A set $R$ of vertices in a graph will be called a \emph{reach} if it is a maximal reachable set; in
other words, $R$ is a reach if $R = R(i)$ for some $i$ and there is no $j$ such that
$R(i)\subsetneq R(j)$. Since our graphs all have finite vertex sets, such maximal
sets exist
and are uniquely determined by the graph. For each reach $R_i$ of a graph, we define the \emph{exclusive
part} of $R_i$ to be the set $H_i = R_i \backslash \cup_{j\neq i}R_j$. Likewise, we define the
\emph{common part} of $R_i$ to be the set $C_i = R_i \backslash H_i$.
\label{def:reaches}
\end{definition}

Thus, a reach is a maximal unilaterally connected set. A graph will typically contain more than one reach.
Note that, by definition, the pairwise intersection of two exclusive sets is empty:
$H_i\cap H_j=\emptyset$ if $i\neq j$. The common sets can, however, intersect. Note further
that each reach $R$ contains at least one vertex $r$ such that its reachable set $R(r)$ equals
the entire reach. Such a vertex is called a \emph{root} of $R$. By definition, any root must
be contained in the exclusive part of its reach.

\begin{definition} Let $G$ be a digraph. Then each reach $R_i$ of $G$ contains a set of roots $B_i$.
The set $B_i$ is called the \emph{cabal} of $R_i$ and is contained in $H_i$.
(A cabal consisting of 1 vertex
is usually called a leader \cite{flock2}.)
\label{def:root-set}
\end{definition}

Dynamical processes on $G$ are most often defined using a version of the Laplacian operator.
In this paper we find it convenient to define these Laplacians in terms of a normalized adjacency
matrix $Q$ obtained by dividing each row by its sum.
This is straightforward, unless $Q$ contains a row $\{Q_{ij}\}_{j=1}^n$ that consists of entirely of zeroes.
One possible solution is to replace $Q_{ii}$ with a 1. We will denote the resulting matrix also by $Q$.
Another possibility is to replace each entry of the row with $1/n$. The resulting matrix
will also be called $Q$. This variation is often called \emph{teleporting}.
The reason for that name will be made clear in Section \ref{chap:random-walks}. Define the
\emph{in-degree} matrix $D$ is the $n\times n$ diagonal matrix whose $i$-th diagonal entry equals the  sum
of the $i$-th row of $Q$. Then, $S\equiv D^{-1}Q$. If we used ``teleporting", we will denote
the resulting matrix by $S_t$. Both $S$ and $S_t$ are \emph{row stochastic} matrices, that is:
they are non-negative and have row sum one.
We can now list three commonly used Laplacians together with their names:
\begin{enumerate}
\item the \emph{combinatorial Laplacian}  $L=D-Q$;
\item the \emph{random walk (rw) Laplacian}  ${\cal L} = I-S$;
\item the \emph{rw Laplacian with teleporting}  ${\cal L}=I-S_t$.
\end{enumerate}
All these Laplacians share the fact that they can be written as follows.

\begin{definition} \cite{CV} We define the following (non-symmetric) Laplacian matrix:
\bsenn
M\equiv D-DS \;,
\esenn
where $D$ is non-negative on the diagonal and zero elsewhere, and $S$ is a row stochastic matrix.
$M$ has non-negative diagonal entries and non-positive off-diagonal entries and any row-sum is zero.
\label{def:laplacians}
\end{definition}

There are many other conventions that may give different Laplacians for a given graph. For example,
if $Q$ is \emph{irreducible}, there is a leading all-positive eigenvector $r$ with (real) eigenvalue
$\lambda$. Let $R$ be the diagonal matrix whose diagonal entries equal $r_i$.
Then $S=\lambda^{-1}R^{-1}QR$ is the \emph{stochasticization} \cite{Boyle} of $Q$.
If $Q$ is undirected (symmetric), then the \emph{normalized} adjacency is $S=D^{-\frac12}Q D^{-\frac12}$.
This normalized matrix has the advantage that it is symmetric if and only if $Q$ is symmetric.
More sophisticated
definitions that give symmetric Laplacians for digraphs can be found in \cite{Chung} and \cite{ChZh}.
We will, however, not pursue these in this article.

The statement that $i$ is in the same strong component as $j$ is an equivalence relation on
the vertices of the graph $G$. Thus the vertices can be partitioned into strong components.
In turn, this generates a partial order on strong components $S_i\subset V$. Namely,
$S_1<S_2$ if $S_2$ is reachable from $S_1$.
The fact that a partial order can be extended to a total order implies \cite{bang} that the
strong components themselves can be ordered in such a way
that the adjacency matrix becomes block triangular, in such a way that each diagonal block
is the adjacency matrix of the subgraph of $G$ by induced by one of the strong components. Thus, the
spectrum of $Q$ in is one-to-one correspondence with the union of the spectra of the strong components
and gives no information of how the strong components are linked together. Even though
this is not entirely true for the eigenvalues of the various \emph{Laplacians}, we take this as
a cue to study the \emph{eigenvectors} of the Laplacians, in particular those that correspond
to the eigenvalue 0. We will see that the structure of these eigenvectors is intimately related
to how the strong components --- and, indeed, the unilateral components or reaches --- are
connected to one another.

Next we define two important dynamical processes on digraphs, diffusion and consensus.
we consider them from the start as dual to each other.

\begin{definition}
Let ${\cal L}$ be the rw Laplacian (with or without teleporting). \\
i) The \emph{consensus problem} on $G$ is the differential equation $\dot x=-{\cal L}x$.\\
ii) The \emph{diffusion problem} on $G$ is the differential equation $\dot p= -p {\cal L}$.\\
From hereon, $x$ is a column vector and $p$ is a row vector.
\label{def:diffusion-consensus1}
\end{definition}

If we discretize both these problems with time-steps of size 1, we get (noting that $-{\cal L}=S-I$)
\benn
\dot x = (S-I)x  &\longrightarrow &  x^{(n+1)}=Sx^{(n)} \;,\\
\dot p = p(S-I)  &\longrightarrow &  p^{(n+1)}=p^{(n)}S \;.\\
\eenn
This inspires our next definition.
\begin{definition}
Let $S=I-{\cal L}$ where ${\cal L}$ is the rw Laplacian (with or without teleporting).
i) The \emph{discrete consensus problem} on $G$ is the difference equation $x^{(n+1)}=Sx^{(n)}$.\\
ii) The \emph{discrete diffusion problem} or \emph{random walk} on $G$ is the difference
equation $p^{(n+1)}=p^{(n)}S$.
\label{def:diffusion-consensus2}
\end{definition}
\noindent Note that the second process is also known as a discrete-time finite-state Markov chain \cite{lev}.

In the above definition, $S$ can be either the matrix $D^{-1}Q$ or $D^{-1}Q_t$ that we defined
earlier. The difference is that in the former case, if the walker reaches
a vertex with in-degree zero, it will stay there forever, while in the latter case, the
walker will be sent to an arbitrary vertex (with uniform distribution) in the graph.
Hence the annotation ``with teleporting" for the corresponding Laplacian. Note that the random walker
moves in the direction \emph{opposite} the direction of the flow of information.

The (continuous) consensus problem can be characterized by two requirements. First,
we require that ${\bf 1}$, the vector of all ones, is an equilibrium.
Second, if ${\bf 1}_{\{i\}}$ is the vector that is 1 on the $i$th vertex and 0 everywhere else,
then $-{\cal L}{\bf 1}_{\{i\}}$ must be non-negative except on the $i$th vertex.
This requirement indicates that $i$ pulls other vertices in \emph{its} direction.
It is easy to see that the first requirement is equivalent to saying that ${\cal L}$ has row sum
zero, while the second indicates that off-diagonal components of ${\cal L}$ are non-positive.

A similar characterization can be given for the diffusion problem. First, we need this definition.

\begin{definition}
A probability vector or a (discrete) \emph{measure} on $G$ is a {row-vector}
in $\R^{V}$ such that
for all $i$, $p(i)\geq 0$ and $\sum_i\,p(i)=1$. The \emph{support}, $\emph{supp}(p)$, of the measure
$p$, is the set of vertices on which $p$ takes a positive value.
\label{def:prob.vector}
\end{definition}

The characterization of the diffusion problem is as follows. First, total probability is conserved, or
$\sum_i p_i=0$. Second, all components except the $i$th of $-{\bf 1}_{\{i\}}^T{\cal L}$ are non-negative,
because probability streams from vertex $i$ to other vertices. One immediately sees that this gives the same
requirements on ${\cal L}$. Similar characterizations can be given for the discrete equivalent
of these problems.

Thus the discrete versions of each process can be treated in pretty much the same way.
In the interest of brevity, we will limit the remainder of the exposition
to considering only the discrete diffusion problem and the continuous consensus problem.
The other two problems can be treated in the same way and with the same conclusions.
A more rigorous treatment of the discrete/continuous distinction is deferred to the appendix.

\begin{centering}
\section{The Right Kernel of $D-DS$}
\label{chap:prior}
\end{centering}
\setcounter{figure}{0} \setcounter{equation}{0}

\noindent In this section, we analyze the right kernel of the Laplacian.

\begin{lemma}
(i): There are \emph{no} edges from the complement $H_i^c$ of an exclusive set to $H_i$.\\
(ii): There are \emph{no} edges from the complement $B_i^c$ of a cabal to $B_i$.\\
(iii): The graph induced by a cabal $B_i$ is strongly connected.
\label{lem:no-incoming}
\end{lemma}

\begin{proof} (i): Suppose that for $j\neq i$,  $v\in R_j$ and $vu \in E$. Then by the definition
of reach, $u\in R_j$. Thus, by the definition of exclusive set, $u\not\in H_i$.\\
(ii): Suppose $u \in B_i$ and $vu \in E$. By the definition of cabal, the entire reachable
from u and therefore from $v$. So $v\in B_i$.\\
(iii): Suppose $v$ and $w$ in $B_i$ and there is no path $v\rightsquigarrow w$ in $B_i$. Then
$v$ cannot be a root, which is a contradiction.
\end{proof}

\noindent
\begin{theorem} \cite{CV}
Suppose $M = D-DS$, where $D$ is a nonnegative $n \times n$ diagonal matrix
and $S$ is (row) stochastic. Suppose $G$ has $k$ reaches, denoted $R_1$ through $R_k$,
where we denote the exclusive and common parts of each $R_i$ by $H_i$, $C_i$ respectively.
Then the eigenvalue 0 has algebraic and geometric multiplicity $k$ and the
kernel of $M$ has a basis $\gamma_1$, $\gamma_2$, ... $\gamma_k$ in $\R^n$
whose elements satisfy:\\
(i) $\gamma_i(v) = 1$ for $v \in H_i$;\\
(ii) $\gamma_i(v) \in (0, 1)$ for $v\in C_i$;\\
(iii) $\gamma_i(v) = 0$ for $v \not \in R_i$;\\
(iv) $\sum_{i} \gamma_i = {\bf 1}_n$ (the all ones vector).
\label{thm:right-null}
\end{theorem}

\begin{theorem} \cite{agaev}, \cite{CV}
Any nonzero eigenvalue of a Laplacian matrix of the form $D-DS$, where
$D$ is nonnegative diagonal and $S$ is stochastic, has (strictly) positive real part.
\label{thm:positive}
\end{theorem}

The definitions of reach, cabal, exclusive part, and common part are illustrated in Section
\ref{chap:examples}. Here we wish to note that the reaches of even a large graph $G$ are relatively
easy to compute by first computing the condensation digraph $SC(G)$. This is the graph whose vertices
are the strong components $S_i$ of $G$ and $S_i\rightarrow S_j$ if there is an edge from $S_i$ to $S_j$.
Let $V_0$ be the vertices of $SC(G)$ with in-degree
0. Since $SC(G)$ has no directed cycles, there must be at least one such vertex.  Any such vertex
$v$ represents a cabal in $G$ by Lemma \ref{lem:no-incoming}.  We can then use breadth first
search from $v$ to find the reach of $v$ in $SC(G)$ and finally recover a reach in $G$.

Theorem \ref{thm:right-null} also shows that the underlying unweighted graph of $G$ gives a great
deal of information about the basis of the kernel. In particular, it determines $\gamma_i$ on
$H_i$ and on the complement of $R_i$. However, the exact value of $\gamma_i(v)$ if $v$ is a
vertex  in a common part is only determined once the weights on the edges are fixed.

\begin{centering}
\section{Random Walks on Directed Graphs}
 \label{chap:random-walks}
 \end{centering}
\setcounter{figure}{0} \setcounter{equation}{0}

\noindent In this section, we analyze random walks on $G$ and analyze its invariant measure in terms of the left kernel
of the Laplacian of $G$.

\begin{definition} The probability vector $p$ is an \emph{invariant probability measure}
(or a \emph{stationary distribution}) for $T$ if $pS=p$.
$K\subseteq V(G)$ is a forward invariant set under $T$ if $\emph{supp}(p)\subseteq K$ implies
$\emph{supp}(pS)\subseteq K$.
\label{def:random-walk}
\end{definition}

\begin{lemma}
Given a random walk random walk $T$, every exclusive set $H_i$ and its cabal $B_i$ are forward
invariant sets under $T$.
\label{lem:forward-invariant}
\end{lemma}

\begin{proof} A walker leaving $C_i$ means that the graph must have an edge into $C_i$.
This contradicts Lemma \ref{lem:no-incoming}. The same holds for $B_i$. \end{proof}

\begin{theorem} Let $G$ be a weighted digraph with Laplacian ${\cal L}$ and $k$ reaches.
The probability that a random walker under $T$ starting at $v$ is absorbed into the cabal
$B_r$ equals $\gamma_r(v)$ (defined in Theorem \ref{thm:right-null}).
\label{theo:walker}
\end{theorem}

\begin{proof}
Let $q(j)$ be the probability that a random walker starting at $j\in V$ reaches
$B_m$ for some fixed $m$. (Note that $q$ is not a probability vector.)
Then $q:V\rightarrow [0,1]$ is well-defined and
is constant in time. Since, by Lemma \ref{lem:forward-invariant}, $B_m$ is forward invariant,
$q(j)$ is also equal to the probability that the walker starting at $j$ ends up and stays in $B_m$.

The probability $q(j)$ concerns the future (under $T$) of the walker on $j$.
Therefore it is equal to the appropriately weighted average of $q(i)$ of $j$'s
{successors under $T$}. Thus from Definition \ref{def:diffusion-consensus2},
\bsenn
q(j)=\sum_{i}\,\mbox{prob}(j\rightarrow i)q(i)= \sum_i\,S_{ji}q(i)\;.
\esenn
From this we conclude $q=Sq$, which is equivalent to ${\cal L}q=0$. Thus $q$ is in the kernel
of ${\cal L}$, and therefore
\bsenn
q(j)=\sum_m\,\alpha_m\gamma_m(j)\;.
\esenn
By Lemma \ref{lem:forward-invariant}, if $j$ is a vertex in $B_r$, then $q(j)=1$.
Also if $j$ is in any $B_m$ with $m\neq r$, then $q(j)=0$. Since the cabals are disjoint,
we must have that $\alpha_r=1$ and $\alpha_m=0$ if $m\neq r$. \end{proof}

We remark that it follows from Theorem \ref{thm:right-null}(iv) that the probability is one that a walker
will be absorbed in $B_r$ for some $r$.

\begin{lemma} Let $G$ be a digraph that has a reach $R$, consisting of an exclusive part $H$
which contains a cabal $B$, and a common part $C$. Under the random walk $T$ on $G$,
there is a unique invariant measure $p$ with support in $R$. Furthermore,
\emph{supp}$(p)$ equals $B$.
\label{lem:no-commoners}
\end{lemma}

\begin{proof} Consider a reach $R$ with its cabal $B$ and denote
the vertex set $R\backslash B$ by $Y$ and the vertex set $V\backslash R$ by $Z$.
Since directed paths in $G$ cannot leave the reach $R$,
we have $S_{ZB}=S_{ZY}=0$. By Lemma \ref{lem:no-incoming}, $S_{BY}=S_{BZ}=0$. So
\bsenn
S=\begin{pmatrix} S_{BB} & {\bf 0} & {\bf 0}\\ S_{YB} & S_{YY} & S_{YZ}\\
{\bf 0} & {\bf 0} & S_{ZZ} \end{pmatrix}\;.
\esenn
We solve for $p$ in $pS=p$, where $p=(a_B,b_Y,c_Z)$ and assume $c_Z=0$. This gives
\bsenn
a_B S_{BB}+b_Y S_{YB}=a_B \; , b_Y S_{YY}=b_Y
 \logand  b_Y S_{ZY}+c_Z S_{ZZ}=c_Z \;.
\esenn
The proof of Theorem 2.7 in \cite{CV} establishes that the spectral radius of $S_{YY}$ is
strictly less than 1. Thus the middle equation can only be satisfied if $b_Y=0$.
Since $B$ is a strong component, $S_{BB}$ is irreducible. It follows directly from the Perron-Frobenius
theorem (see \cite{Boyle,Stern}) that the eigenvalue 1 is simple and its associated eigenvector is
strictly positive. \end{proof}

\begin{theorem} Let $G$ be a graph with Laplacian ${\cal L}=I-S$ with $k$ reaches.
Then the eigenvalue 0 of ${\cal L}$ has algebraic and geometric
multiplicity $k$. The left kernel of ${\cal L}$ has a basis
$\bar \gamma_1$, $\bar\gamma_2$, ... $\bar\gamma_k$ in $\R^n$
whose elements satisfy:\\
(i) For all $i\in \{1,\cdots k\}$ and all $v\in \{1,\cdots n\}$: $\bar\gamma_i(v)\geq 0$;\\
(ii) supp$(\bar \gamma_i)=B_i$;\\
(iii) $\sum_v\,\bar\gamma_i(v) =1$;\\
(iv) The vectors $\{\bar \gamma_i\}_{i=1}^k$ are orthogonal.
\label{thm:left-nullspace}
\end{theorem}

\begin{proof} The first statement (the multiplicity of $0$) is in fact part of Theorem
\ref{thm:right-null}. For each reach $R_i$, Lemma \ref{lem:no-commoners} gives an invariant
measure $\bar\gamma_i$ satisfying (i), (ii), and (iii). Since
$\bar\gamma_iS=\bar\gamma_i$, we have that $\bar\gamma_i$ is a left null vector of ${\cal L}$.
These $k$ vectors are orthogonal, because the sets $B_i$ are mutually disjoint. \end{proof}

\begin{corollary} Let $G$ be a graph with Laplacian $L=D-DS$ with $k$ reaches. Then the left kernel of
$L$ has a basis $\left\{\bar \gamma_iD^{-1}\right\}_{i=1}^k$, where the $\{\bar \gamma_i\}_{i=1}^k$
are as in Theorem \ref{thm:left-nullspace}.
\label{cory:left-nullspace}
\end{corollary}

\begin{proof} Construct the basis $\bar \gamma_1$, $\bar\gamma_2$, ... $\bar\gamma_k$
of the left kernel of $I-S$. Then it is easy to see
that the $\bar \gamma_i D^{-1}$ form a basis of the left kernel of $D-DS$.       \end{proof}

\begin{centering}
\section{The Left and Right Kernels of ${\cal L}$}
\label{chap:left-and-right}
\end{centering}
\setcounter{figure}{0} \setcounter{equation}{0}

\noindent We start this section with a general lemma and a definition that will be used at the end of the section.
Subsequently, we establish the dual relationship between the equilibria of the random walk
and those of the consensus problem.  In this context, the duality is very natural. The matrix
is multiplied from the \emph{left} by probability vectors (measures) and from the right
by vectors whose Laplacian is zero. The latter are also known as \emph{harmonic vectors or functions}
and have many other applications in graph theory.  Two notable examples focused on undirected graphs are \cite{lov} and \cite{bollob}.  The latter reference also discusses the duality mentioned above in the context of electrical networks.  At the end of the section, we briefly
note that the duality is independent of the discrete or continuous nature of the processes discussed.
This has also been observed by other authors, for example \cite{mas2} and \cite{mirz}.

\begin{lemma} Let $A$ be an $n\times n$ matrix. There are bases of (generalized) right eigenvectors
 (columns), $\{\eta_i\}_{i=1}^n$, and of (generalized) left eigenvectors (rows),
 $\{\bar\eta_i\}_{i=1}^n$, such that the matrices:
\bsenn
\Gamma = \left(\begin{array}{cccc} \eta_1& \eta_2 & \cdots & \eta_n\end{array}\right)
 \logand
\bar \Gamma = \left(\begin{array}{c} \bar\eta_1\\ \bar\eta_2\\ \vdots \\ \bar\eta_n\end{array}\right)
\esenn
are inverses of one another. We will assume the basis-vectors ordered according to ascending eigenvalue.
\label{lem:eigenvecs}
\end{lemma}

\begin{proof} This follows directly from the Jordan Decomposition Theorem \cite{HJ}.
Let $J$ be the Jordan normal form of $A$. Then that theorem tells us that there
is an invertible matrix $\Gamma$ such that
$A\Gamma=\Gamma J$ or $\Gamma^{-1} A= J \Gamma^{-1}$. Right multiply the first equation
by the standard column basis vector ${\bf 1}_{\{i\}}$ to show that the $i$th column of $\Gamma$
is a generalized right eigenvector.
Left multiply by ${\bf 1}_{\{i\}}^T$ to see that the $i$th row of $\Gamma^{-1}$ is a generalized
left eigenvector. \end{proof}

\begin{definition} Let $G$ have $k$ reaches. Let $\{\gamma_i\}_{i=1}^k$ be the (column) vectors of Theorem
\ref{thm:right-null} and $\{\bar\gamma_i\}_{i=1}^k$ the (row) vectors of Theorem \ref{thm:left-nullspace}.
Define
\bsenn
\Gamma^0 = \left(\begin{array}{cccc} \gamma_1& \gamma_2 & \cdots & \gamma_k\end{array}\right)
 \logand
\bar{\Gamma}^0 = \left(\begin{array}{c} \bar\gamma_1\\ \bar\gamma_2\\ \vdots \\ \bar\gamma_k\end{array}\right)\;.
\esenn
\label{defn:Gamma0}
\end{definition}

\begin{theorem} Let $G$ be a graph with random walk laplacian ${\cal L}=I-S$ with $k$ reaches $\{R_i\}_{i=1}^k$.
Let $p^{(0)}$ be an initial measure (at $t=0$). Then
\bsenn
\lim_{\ell\rightarrow \infty}\frac1\ell \sum_{i=0}^{\ell-1}\,T_*^ip^{(0)} =
\sum_{m=1}^k\,\left(p^{(0)}\gamma_m \right)\,\bar \gamma_m\;.
\esenn
Note that $(p^{(0)}\gamma_m)$ is a scalar. (Recall that $p$ is a row vector and $\gamma$ is a column vector.)
\label{thm:equil-random-walker}
\end{theorem}

\begin{proof} Given an initial measure $p^{(0)}$, then Theorem \ref{theo:walker} implies that
$\alpha_r\equiv p^{(0)}\gamma_r$ is the probability to be absorbed in the cabal $B_r$.
By Lemma \ref{lem:no-incoming}, the graph induced by $B_i$ is strongly connected.
This is equivalent with the matrix $S_{B_iB_i}$ being \emph{irreducible}.
The Perron-Frobenius theorem (see \cite{Boyle,Stern}) implies that $S_{B_iB_i}$ satisfies:\\
{\bf 1.} eigenvalue 1 has algebraic and geometric multiplicity 1, and\\
{\bf 2.} there are possibly other eigenvalues that are roots of 1, all with multiplicity 1, and\\
{\bf 3.} all eigenvalues not listed under cases 1 or 2, have absolute value less than 1.\\
The eigenvector corresponding to case 1 is the unique invariant measure in $B_r$.
For the eigenvector $p$ of case 2, we have
that $T_*^ip=\lambda^i p$ with $\lambda$ a root of unity. This implies that
\bsenn
\lim_{\ell\rightarrow \infty}\frac1\ell \sum_{i=0}^{\ell-1}\,T_*^ip=0\;.
\esenn
The same holds for case 3.

Thus for each $r$, the probability is $\alpha_r$ that a random walker is absorbed in $B_r$.
Thus for large $i$, $T_*^ip$ is a combination of
invariant measures in $B_r$, a finite number of periodic measures, and a remainder that
decays exponentially. The averaging cancels the periodic (case 2) and decaying (case 3) parts. Hence,
\bsenn
\lim_{\ell\rightarrow \infty}\frac1\ell \sum_{i=0}^{\ell-1}\,T_*^ip^{(0)} =
\sum_{m=1}^k\,\alpha_r\,\bar \gamma_m =
\sum_{m=1}^k\,\left(p^{(0)}\gamma_m \right)\,\bar \gamma_m\;.
\esenn
\end{proof}

\begin{corollary} Assume the hypotheses of Theorem \ref{thm:equil-random-walker}.
Suppose that in addition the following is true:
for each cabal $B_i$, there is a $k_i$ so that the $k_i$-th power of $S_{B_iB_i}$ is
strictly positive, then
\bsenn
\lim_{\ell\rightarrow \infty}\,T_*^\ell p^{(0)} =
\sum_{m=1}^k\,\left(p^{(0)}\gamma_m \right)\,\bar \gamma_m\;.
\esenn
\label{cor:equil-random-walker}
\end{corollary}

\begin{proof} The proof of this statement is exactly the same as the proof of Theorem
\ref{thm:equil-random-walker}, except that now each matrix $S_{B_iB_i}$
is \emph{primitive}. Primitivity guarantees that case 2 of that proof does not occur.
For a (generalized) eigenvector $p$ as in case 3, $T_*^ip$ converges to 0. Hence the
averaging is not necessary. \end{proof}

\begin{theorem} Let $G$ be a graph with Laplacian ${\cal L}=I-S$ with $k$ reaches $\{R_i\}_{i=1}^k$.
Then the solution of the consensus problem of Definition \ref{def:diffusion-consensus1} with initial condition
(at $t=0$) $x^{(0)}$ satisfies
\bsenn
\lim_{t\rightarrow \infty}\, x^{(t)} =
\sum_{m=1}^k\,\left(\bar\gamma_m x^{(0)}\right)\,\gamma_m\;.
\esenn
Note that $(\bar\gamma_m x^{(0)})$ is a scalar.
\label{thm:consensus}
\end{theorem}

\begin{proof} The right and left eigenvectors of ${\cal L}$ defined in Theorems \ref{thm:right-null}
and \ref{thm:left-nullspace} satisfy $\bar \gamma_i \gamma_j=\delta_{ij}$ for $i,j\in\{1,\cdots k\}$.
According to Lemma \ref{lem:eigenvecs}, one can extend theses sets of vectors to
dual bases of right and left (generalized) eigenvectors $\{\gamma_\ell\}_{\ell=1}^n$ and
$\{\bar \gamma_\ell\}_{\ell=1}^n$ such that for $i\leq k$ and $\ell>k$
\bsenn
\bar \gamma_i \gamma_\ell=\bar \gamma_\ell \gamma_i = 0\;.
\esenn
Let $\lambda_i$ be the $i$th eigenvalue associated with the (generalized) eigenvector $\gamma_i$ (or
$\bar \gamma_i$). Now we consider the consensus problem with initial condition $x^{(0)}=y^{(0)}+z^{(0)}$ where
\bse
x^{(0)}=y^{(0)}+z^{(0)} \where
y^{(0)}=\sum_{i=1}^k\,\alpha_i \gamma_i  \logand
z^{(0)}=\sum_{i=k+1}^n\,\alpha_i \gamma_i \;.
\label{eq:initcondn}
\ese

From the standard theory of linear differential equations (see, for example, \cite{arnold}), one
easily derives that the general solution of the consensus problem is given by
\bse
x^{(t)}=e^{-{\cal L} t}x^{(0)}=\sum_{i=1}^n\, \beta_i \gamma_i e^{-\lambda_i t}p_i(t)\;,
\label{eq:soln-consensus}
\ese
where $\beta_i$ are constants. Here $p_i(t)$ are polynomials whose degree is less than the size of
the corresponding Jordan block. Furthermore, if $\gamma_i$ is an actual eigenvector (not a generalized one),
then $\alpha_i=\beta_i$.

By Theorem \ref{thm:right-null}, we have that in equation (\ref{eq:soln-consensus}), $\lambda_i=0$
for $i\in \{1,\cdots k\}$. Also, the eigenvalue zero has only trivial Jordan blocks and so
for $i\in \{1,\cdots k\}$, $p_i=1$ and $\beta_i=\alpha_i$.
By Theorem \ref{thm:positive}, we have that all the other terms (for $i>k$) converge to zero. Therefore,
using equation (\ref{eq:initcondn}),
\bsenn
\lim_{t\rightarrow \infty}\, x^{(t)} =
\sum_{m=1}^k\,\alpha_m\,\gamma_m =y^{(0)}\;.
\esenn

Next, we determine the $\alpha_i$. Let $\Gamma$ be the matrix whose $i$th column equals $\gamma_i$.
The previous equation implies that ($\gamma_i$ being he columns of $\Gamma$)
\bsenn
\Gamma\left(\sum_{i=1}^k\,\alpha_i {\bf 1}_{\{i\}} \right)=y^{(0)}  \Longrightarrow
\sum_{m=1}^k\,\alpha_m {\bf 1}_{m} = \bar \Gamma y^{(0)}\;.
\esenn
With the definition of $x^{(0)}$ and first paragraph of this proof, this implies that
\bsenn
\alpha_m=\bar \gamma_m y^{(0)} =\bar \gamma_m x^{(0)}\;,
\esenn
from which the result follows.
\end{proof}
\noindent An alternate proof of this result can be found in \cite{mirz}.

There is a striking way to express the duality between consensus and diffusion. In Theorem
\ref{thm:consensus},  the factor within the parentheses in $\left(\bar\gamma_m x^{(0)}\right)\,\gamma_m$
is a real number. Since multiplication of a real and vector is commutative, it can be written
as $\gamma_m\,\left(\bar\gamma_m x^{(0)}\right)$. But multiplication is associative,
and thus this is equal to $\left(\gamma_m\,\bar\gamma_m\right)\, x^{(0)}$.
The same reasoning works for the corresponding expression in Theorem \ref{thm:equil-random-walker}.
The notation becomes more compact upon observing that
$\left[\sum_{m=1}^k\gamma_m \,\bar \gamma_m\right]=\Gamma^0\bar{\Gamma}^0$. Thus we obtain the following.

\begin{corollary} Let $G$ be a graph with Laplacian ${\cal L}=I-S$ with $k$ reaches $\{R_i\}_{i=1}^k$.
Then the solution of the random walker with initial condition $p^{(0)}$ and of the consensus problem
with initial condition $x^{(0)}$, respectively, satisfy
\bsenn
\lim_{\ell\rightarrow \infty}\frac1\ell \sum_{i=0}^{\ell-1}\,T_*^ip^{(0)} =
p^{(0)}\,\Gamma^0\bar{\Gamma}^0  \logand
\lim_{t\rightarrow \infty}\, x^{(t)} =\Gamma^0\bar{\Gamma}^0\,x^{(0)}\;.
\esenn
Equivalently, $\lim_{\ell\rightarrow \infty}\frac1\ell \sum_{i=0}^{\ell-1}\,S^j=\lim_{t\rightarrow \infty}\,e^{-{\cal L}t}=\Gamma^0\bar{\Gamma}^0$.
\label{cory:walker+consensus}
\end{corollary}

Of course, the same result also holds for the discrete version of the consensus and the
continuous version of diffusion mentioned in Definitions \ref{def:diffusion-consensus1} and \ref{def:diffusion-consensus2}.  Note that the matrix $e^{-{\cal L} t}$ is well studied in the context of a continuous-time Markov chains where it is called the \emph{heat kernel}.  We refer the reader to \cite{lev} for more details.

\begin{centering}
\section{Application to Ranking Algorithms}
\label{chap:ranking}
\end{centering}
\setcounter{figure}{0} \setcounter{equation}{0}

\noindent In this section, we derive the pagerank algorithm using the consensus problem. This
explanation is the ``dual" of the usual one, which is in terms of the random walk.
We also show that pageranks ``with teleporting" can be easily expressed in terms
of pageranks ``without teleporting".

For a graph with vertex set $V$ and $U\subseteq V$, let ${\bf 1}_U$ denote the (column) vector that
has the value 1 on $U$ and 0 elsewhere. We set ${\bf 1}_V\equiv {\bf 1}$ (as before).
The second part of Corollary \ref{cory:walker+consensus} tells us that in the consensus
problem the final displacements when the initial condition is ${\bf 1}_{\{v\}}$ are equal to
$\left[ \sum_{m}\,\gamma_m \bar \gamma_m)\right]\,{\bf 1}_{\{v\}}$. As first noted in \cite{mas}, the \emph{mean} of these displacements over all vertices is a good measure of the influence that vertex $v$ exerts over the long term behavior of the solution to the consensus problem.
This naturally leads to the following definition.

\begin{definition} For a graph $G$ with $n$ vertices, the \emph{influence} of a vertex $v$ is given by
\bsenn
I(v)=\frac{{\bf 1}^T}{n}\,\Gamma^0 \bar{\Gamma}^0\,{\bf 1}_{\{v\}}\;.
\esenn
The \emph{influence vector} $I$ is the row vector whose $j$th component equals $I(j)$ or
\bsenn
I=\frac{{\bf 1}^T}{n}\,\Gamma^0 \bar{\Gamma}^0\;.
\esenn
\label{def:influence}
\end{definition}

From this we see that the influence of $v$ equals the average of the $v$th column of the
matrix $\Gamma^0 \bar{\Gamma}^0$. It is nonnegative. It is positive if and
only if some $\bar \gamma_m$ is non-zero at $v$, that is: if $v$ is part of
a cabal. The sum of all influences equals 1. Thus the influence vector is a probability measure.
It is straightforward to see that the influence of a subset $U$ of vertices equals:
$\frac{{\bf 1}^T}{n}\,\left[ \sum_{m}\,\gamma_m \bar{\gamma}_m\right]\,{\bf 1}_{U}$.

The fact that $I(v)>0$ only if $v$ is part of a cabal is not entirely realistic for networks
such as the internet.
A page may contain information even though it has both in- and out-links. That information, of
course, had to be put there by some other entity in the first place. Following \cite{mas}, one could say that, in fact,
each page $v$ has an associated (meta) page $b_v$ together with an edge $b_v\rightarrow v$.

\begin{definition} Let $G$ be a graph with $n$ vertices and stochastic weighted adjacency matrix $S$.
The \emph{extended graph} $E_\alpha[G]$ is defined as the graph $G$ together with one new vertex
$b_v$ and one new edge $b_v\rightarrow v$ of weight $\alpha>0$ for each $v\in V$.
\label{def:extended graph}
\end{definition}

The new graph $E_\alpha[G] $ has exactly $n$ reaches $\{\tilde{R}_i\}_{i=1}^n$ and $R_i$
the new vertex $b_i$ as leader. The natural choice for the weight $\alpha$ would appear
to be 1, but in actual applications it tends to have different values. So we keep it as a
(positive) parameter. In what follows we need to distinguish vectors in $\R^{2n}$ from those in $\R^n$.
To do that we mark the former with a tilde. The same device is used to distinguish $2n \times 2n$
matrices from $n\times n$ matrices. We label the vertices of $E_\alpha[G]$ in the following order:
$\{b_1,\cdots b_n, 1,\cdots n\}$. The first $n$ vertices are the leaders and this set
will be denoted by $B$. The second set of $n$ vertices we continue to refer to as $V$.

The Laplacian and the random walk Laplacian associated with $E_\alpha[G]$ are
\bsenn
\tilde{L} =\begin{pmatrix}  {\bf 0} & {\bf 0}\\ -\alpha I & (1+\alpha)I-S \end{pmatrix}
 \Longrightarrow\
\tilde{{\cal L}} =\begin{pmatrix}  {\bf 0} & {\bf 0}\\
 -\frac{\alpha}{1+\alpha} I & I-\frac{1}{1+\alpha} S \end{pmatrix}\;.
\esenn
The normalized Laplacian with teleportation is given in the same way, except that we use $S_t$
instead of $S$. We are now ready to discuss pagerank.  We begin with an alternate formulation.  In
Corollary \ref{cor:pagerank}, we will show how this is equivalent to the standard treatment
(as found, for example, in \cite{Stern}).

\begin{definition} For a graph $G$ with $n$ vertices, the \emph{pagerank} of a vertex $v$ is given by
\bsenn
\wp(v)=2\tilde{I}(b_v)-\frac1n\;.
\esenn
Here, $\tilde{I}$ is the influence of $b_v$ in the extended graph $E_\alpha[G]$.
The \emph{pagerank vector} $\wp\in \R^n$ is the row vector whose $j$th component equals $\wp(j)$.
\label{def:pagerank}
\end{definition}

\begin{theorem} Let $G$ be a graph with vertex set $V$, $|V|=n$, and Laplacian ${\cal L}=I-S$ or
${\cal L}=I-S_t$. The pagerank vector of Definition \ref{def:pagerank} is the unique probability
measure on $\R^n$ that satisfies
\bsenn
\wp = \dfrac{\alpha}{n}{\bf 1}^T(\alpha I+{\cal L})^{-1}\;.
\esenn
\label{thm:pagerank}
\end{theorem}

\begin{proof} From the comments after Definition \ref{def:influence}, we know that the sum of the
influences equals 1. We know that the influence of non-leaders is zero and so
$\sum_v\,\left(2\tilde{I}(b_v)-\frac1n\right)=2-1=1$. Also $\wp(v)\geq 0$, because the
displacement of $b_v$ in the consensus problem is 1 and the others are non-negative
(Theorem \ref{thm:consensus}).

Sections \ref{chap:prior} and \ref{chap:random-walks} provides us with appropriate bases
$\{\tilde{\bar{\gamma}}_i\}_{i=1}^n$ and $\{\tilde{{\gamma}}_i\}_{i=1}^n$ of the left and right
kernels of $\tilde{\cal{L}}$. From Theorems \ref{thm:left-nullspace} and \ref{thm:right-null} we know that for $1\leq m\leq n$,
\bsenn
\tilde{\bar{\gamma}}_m=({\bf 1}_{\{m\}}^T,{\bf 0}^T)  \logand
\tilde{{\gamma}}_m= \begin{pmatrix}  {\bf 1}_{\{m\}}\\\eta_m \end{pmatrix}\;.
\esenn
With these constraints, $\tilde{\gamma}_m$ is entirely determined by $\tilde{L}\tilde{{\gamma}}_m=0$.
This is equivalent to
\bsenn
(\alpha I+ {\cal L})\eta_m=\alpha {\bf 1}_{\{m\}}\;.
\esenn
Since $\alpha>0$ and ${\cal L}$ has no negative eigenvalues, the matrix $(\alpha I+ {\cal L})$
can be inverted. Thus $\wp(m)$ is the average over $V$ of the displacements $\eta_m$:
\bsenn
\wp(m)=\frac{{\bf 1}^T}{\alpha}\eta_m= \dfrac{\alpha}{n}{\bf 1}^T(\alpha I+{\cal L})^{-1}{\bf 1}_{\{m\}}\;.
\esenn
The result follows immediately from this.
\end{proof}

We use the vectors $\tilde{\bar{\gamma}}_m$ and $\tilde{\bar{\gamma}}_m$ to define the matrices
$\tilde{\Gamma}^0$ and $\tilde{\bar{\Gamma}}^0$ analogous to Definition \ref{defn:Gamma0}. From Definition \ref{def:influence}, we obtain that
\bsenn
\tilde{I}(b_v)=\frac{\tilde{{\bf 1}}^T}{2n}\,\tilde{\Gamma}^0 \tilde{\bar{\Gamma}}^0\,\tilde{{\bf 1}}_{\{b_v\}}=
\frac{\tilde{{\bf 1}}_B^T+\tilde{{\bf 1}}_V^T}{2n}\,\tilde{\Gamma}^0 \tilde{\bar{\Gamma}}^0\,\tilde{{\bf 1}}_{\{b_v\}}\;.
\esenn
The initial condition $\tilde{{\bf 1}}_{\{b_v\}}$ in the consensus problem for $E_\alpha[G]$ displaces
no leaders, except $b_v$ itself by one unit. Thus
$\frac{\tilde{{\bf 1}}_V^T}{2n}\,\tilde{\Gamma}^0 \tilde{\bar{\Gamma}}^0\,\tilde{{\bf 1}}_{\{b_v\}}=\dfrac{1}{2n}$. Putting this together,
we see that
\bsenn
\wp(v)=\frac{\tilde{{\bf 1}}_B^T}{n}\,\tilde{\Gamma}^0 \tilde{\bar{\Gamma}}^0\,\tilde{{\bf 1}}_{\{b_v\}}\;.
\esenn
In other words, the pagerank of the vertex $v$ is in fact the mean over \emph{only} the vertices in
$V$ of the ``old" graph $G$ of the displacements in the consensus problem on $E_\alpha[G] $ caused
by initial condition ${\bf 1}_{\{b_v\}}$. We now turn to a different interpretation of the pagerank.

\begin{corollary} Let $G$ be a graph with vertex set $V$, $|V|=n$, and Laplacian ${\cal L}=I-S$.
The pagerank vector of Definition \ref{def:pagerank} is the probability measure on $\R^n$ that satisfies
\bsenn
\wp=\wp\left( \beta S+\dfrac{1-\beta}{n} J\right) =
\beta \wp S +\dfrac{1-\beta}{n}\,{\bf 1}^T =0\;,
\esenn
where $J$ is the all ones matrix and $\beta\equiv \dfrac{1}{1+\alpha}\in (0,1)$.
\label{cor:pagerank}
\end{corollary}

\begin{proof} Since $\wp$ is a probability measure, we have that ${\bf 1}^T=\wp J$, where $J$ is the
all ones matrix. Substituting this in Theorem \ref{thm:pagerank} and right multiplying by
$(\alpha I + {\cal L})$ gives
\bsenn
\wp\left(\alpha I+{\cal L}-\dfrac\alpha n J\right) = 0\;.
\esenn
Substitute ${\cal L}=I-S$ and divide by $1+\alpha$ to get
\bsenn
\wp\left(I-\dfrac{1}{1+\alpha}S-\dfrac{\alpha}{n(1+\alpha)} J\right) = 0\;.
\esenn
Substituting $\beta$ for $\alpha$ gives the required statement.
\end{proof}

The traditional interpretation of the pagerank (as discussed in \cite{Stern})
is evident in this result.
The pagerank vector is the unique invariant probability measure that results from
applying the random walk with probability $\beta$ and uniform teleporting with
probability $1-\beta$. It is clear that the resulting matrix
$S_{pagerank}\equiv\beta S+\dfrac{1-\beta}{n} J$ is row stochastic and primitive.
Thus standard arguments using the Perron-Frobenius theorem show that the
eigenvalue 1 is simple and all other eigenvalues are strictly smaller. In fact,
since (generalized) eigenvectors of $S$ are eigenvectors of $J$, it follows that all other
eigenvalues have modulus less than $\beta$. The leading eigenvector is strictly positive.
Google's pagerank algorithm takes $\beta=0.85$ \cite{Stern}, and so after roughly 57 iterates,
convergence of $S_{pagerank}^n$ to the pagerank is already accurate to 4 decimal places.
The quick convergence plus the fact that $S_{pagerank}^n$ can be
cheaply computed without using the full matrix, guarantees efficient algorithms for
the computation of the pagerank vector.

The last result describes essentially two ranking algorithms, depending on how $S$
is defined. In the introduction, we gave the definition of $S$ and of $S_t$ with
teleporting. These give rise to pageranks $\wp$ and $\wp_t$. These two ranks
have an interesting relation to one another, essentially that of deck of
cards before and after one shuffle.

To see this, denote the set of vertices that are leaders by $L$ and the rest by $R$.
Thus
\bsenn
S=\begin{pmatrix} S_{LL} & S_{LR}\\ S_{RL} & S_{RR} \end{pmatrix}
\;, \wp=(\wp_L,\wp_R)\;, {\bf 1}^T=({\bf 1}^T_L, {\bf 1}^T_R)\;,  {\textrm etc}.
\esenn

\begin{proposition}With the above notation, we have
\benn
\wp_{t,L}&=&\left(\beta \pi_t+(1-\beta )\right)\wp_L\;,\\
\wp_{t,R}&=& \left(\dfrac{\beta}{1-\beta}\,\pi_t+1\right)\,\wp_R\;,
\eenn
where $\pi_t=\sum_{j\in L}\,\wp_t(j)$.
\label{prop:pagerank}
\end{proposition}

\begin{proof}
Corollary \ref{cor:pagerank} says:
\benn
\wp_L(I-\beta S_{LL}) = \beta \wp_R S_{RL} + \dfrac{1-\beta}{n}\,{\bf 1}_L^T\;,\\
\wp_R(I-\beta S_{RR}) = \beta \wp_L S_{LR} + \dfrac{1-\beta}{n}\,{\bf 1}_R^T\;.\\
\eenn
In the case without teleporting, we get:
\bsenn
S_{LL}=I_{LL}  \logand  S_{LR}=0 \;.
\esenn
With teleporting, this becomes:
\bsenn
S_{LL}=\dfrac 1n J_{LL}  \logand  S_{LR}=\dfrac 1n J_{LR} \;.
\esenn
The former will not receive a subscript, the latter will be denoted with the subscript
``t''. Denote $\pi=\sum_{j\in L}\,\wp(j)$ and $\pi_t=\sum_{j\in L}\,\wp_t(j)$.
Thus, without teleporting, the equations of Corollary \ref{cor:pagerank} become:
\bsenn
\begin{array}{ccl}
\wp_L &=& \dfrac{\beta}{1-\beta}\, \wp_R S_{RL} + \dfrac{1}{n}\,{\bf 1}_L^T\;,\\[0.3cm]
\wp_R &=& \dfrac{1-\beta}{n}\,{\bf 1}_R^T(I-\beta S_{RR})^{-1}\;.\\
\end{array}
\esenn
Since $S$ has spectral radius 1 and $\beta\in (0,1)$, the inverse is well-defined.
Noting that $\wp_{t,L} J_{LR}=\pi_t {\bf 1}_R$ and $\wp_{t,L} J_{LL}=\pi_t {\bf 1}_L$,
the equations \emph{with} teleporting become
\bsenn
\begin{array}{ccl}
\wp_{t,L} &=& \beta \wp_{t,R} S_{RL} +
\dfrac{1}{n}\,\left(\beta\,\pi_t+(1-\beta)\right)\,{\bf 1}_L^T\;,\\
\wp_{t,R} &=& \dfrac{1-\beta}{n}\,\left(\dfrac{\beta}{1-\beta}\,\pi_t+1\right)\,
{\bf 1}_R^T(I-\beta S_{RR})^{-1}\;.
\end{array}
\esenn
This gives the required relation between $\wp_{t,R}$ and $\wp_R$.
Substituting these expressions into $\wp_L$ and $\wp_{t,L}$, respectively, gives:
\bsenn
\begin{array}{ccl}
\wp_{L} &=& \dfrac{\beta}{n}\,{\bf 1}_R^T\,(I-\beta S_{RR})^{-1}\,S_{RL}+
\dfrac{1}{n}\,{\bf 1}_L^T\;.\\[0.3cm]
\wp_{t,L} &=& \dfrac{\beta}{n}\,\left(\beta\,\pi_t+(1-\beta)\right)\,{\bf 1}_R^T\,
{\bf 1}_R^T(I-\beta S_{RR})^{-1}\,S_{RL} + \dfrac{1}{n}\,\left(\beta\,\pi_t+(1-\beta)\right)\,
{\bf 1}_L^T\;.
\end{array}
\esenn
And this yields the relation between $\wp_{t,L}$ and $\wp_L$.
\end{proof}

To express $\wp_t$ completely in terms of $\wp$, we add the following result.

\begin{corollary} With the above notation, we also have
\benn
\begin{array}{ccl}
\pi_t &=& \dfrac{(1-\beta)\pi}{1-\beta \pi}\;.
\end{array}
\eenn
\label{cor:pagerank2}
\end{corollary}

\begin{proof} Summing the coefficients on the left and right hand sides of
the equations of Proposition \ref{prop:pagerank}, and taking into account that
$\sum_{j\in R}\,\wp_(j)=1-\pi$ (and similar for $\pi_t$), we get
\benn
\begin{array}{ccl}
\pi_t &=& \left(\beta \pi_t +(1-\beta)\right)\pi\\[0.3cm]
1-\pi_t &=& \dfrac{1}{1-\beta}\,\left(\beta \pi_t +(1-\beta)\right)(1-\pi)\;.
\end{array}
\eenn
Dividing the first equation by the second gives
$\dfrac{\pi_t}{1-\pi_t}=(1-\beta)\,\dfrac{\pi}{1-\pi}$,
which is equivalent to the statement of the Corollary. \end{proof}

\begin{centering}
\section{A Simple Example}
\label{chap:examples}
\end{centering}
\setcounter{figure}{0} \setcounter{equation}{0}

\noindent In this section, we illustrate the definitions and results with the graph given in Figure \ref{fig:graph-example}.
The unweighted Laplacian matrix equals (row and column $i$ correspond to vertex $i$)
\bsenn
M=  \left( \begin {array}{ccccccc} 0&0&0&0&0&0&0\\ -1&1&0&0&0&0&0\\ 0&0&1&0&-1&0&0\\
0&0&-1&1&0&0&0\\ 0&0&0&-1&1&0&0\\ -1&0&0&0&0&2&-1\\ 0&0&-1&0&0&-1&2\end {array} \right)\;.
\esenn
By ordering the vertices so that we list the exclusive part first, we see that the matrix
$M$ is lower block-triangular. In the exclusive blocks, we list the vertices in the cabal
first, and then the others. The exclusive block itself is then also lower block-triangular.

\begin{figure}[pbth]
\begin{center}
\includegraphics[width=5.0in,height=1.5in]{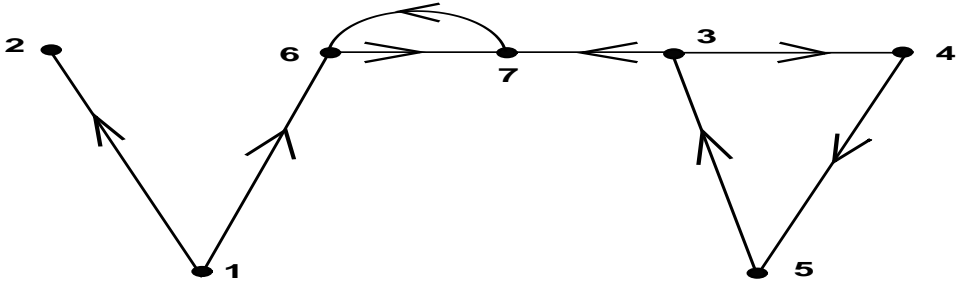}
\caption{\emph{A graph (vertices as labeled) with two reaches $R_1=\{1,2,6,7\}$
and $R_2=\{3,4,5,6,7\}$.
Each reach $R_i$ has cabal $B_i$, exclusive part $H_i$, and common part $C_i$.
According to the definitions in the text: $B_1=\{1\}$, $H_1=\{1,2\}$, $C_1=\{6,7\}$,
$B_2=\{3,4,5\}$, $H_2=\{3,4,5\}$, $C_2=\{6,7\}$.}}
\label{fig:graph-example}
\end{center}
\end{figure}

The algorithm at the beginning of Section \ref{chap:random-walks} gives the stochastic matrix $S$ as
\bsenn
S =  \left( \begin {array}{ccccccc} 1&0&0&0&0&0&0\\ 1&0&0&0&0&0&0\\ 0&0&0&0&1&0&0\\
0&0&1&0&0&0&0\\ 0&0&0&1&0&0&0\\ 1/2&0&0&0&0&0&1/2\\ 0&0&1/2&0&0&1/2&0\end {array} \right)
\logand
{\cal L} =\left( \begin {array}{ccccccc} 0&0&0&0&0&0&0\\ -1&1&0&0&0&0&0\\ 0&0&1&0&-1&0&0\\
0&0&-1&1&0&0&0\\ 0&0&0&-1&1&0&0\\ -1/2&0&0&0&0&1&-1/2\\ 0&0&-1/2&0&0&-1/2&1\end {array} \right)\;.
\esenn
The stochastic matrix with teleporting $S_t$ is nearly the same: the second row entries are replaced
by 1/7.

The bases of the left and right kernels of the Laplacian ${\cal L} = I - S$ of Theorems
\ref{thm:right-null} and \ref{thm:left-nullspace} are, respectively,
\bsenn
\begin{array}{c}
\gamma_1^T  = \left(\begin{array}{ccccccc}1,& 1,& 0,& 0,& 0,& \frac23,&\frac13\end{array}\right)
 \logand
\gamma_2^T  = \left(\begin{array}{ccccccc}0,& 0, & 1, & 1, &, 1,& \frac13 ,&\frac23\end{array}\right) \;,\\[.2in]
\bar\gamma_1= \left(\begin{array}{ccccccc}1, & 0,& 0,& 0, & 0,& 0,& 0 \end{array}\right)
 \logand
\bar\gamma_2=\left(\begin{array}{ccccccc}0,& 0,&\frac13,&\frac13,&\frac13& 0,& 0 \end{array}\right) \;.
\end{array}
\esenn
The general solutions to both the random walk and the consensus problems are given in terms
of these vectors as in Theorems \ref{thm:equil-random-walker} and \ref{thm:consensus}. Note that
for the former, Corollary \ref{cor:equil-random-walker} does not hold.
In this example, the matrix $\left[\frac{1}{n} \sum_{m}\,\gamma_m \bar \gamma_m(v)\right]$ is:
\bsenn
\frac{1}{7\cdot 9} \, \left( \begin {array}{ccccccc} 9&0&0&0&0&0&0\\ 9&0&0&0&0&0&0\\ 0&0&3&3&3&0&0\\
0&0&3&3&3&0&0\\ 0&0&3&3&3&0&0\\ 6&0&1&1&1&0&0\\ 3&0&2&2&2&0&0 \end {array} \right)\;.
\esenn
Thus the influence vector is
\bsenn
\left(\begin{array}{ccccccc}\frac37,& 0,& \frac{4}{21},& \frac{4}{21},& \frac{4}{21}, &0,& 0 \end{array}\right)
\esenn

Next, we solve for the pageranks with $\alpha=1$ (or $\beta=\frac12$). Without teleporting we get
\bsenn
(I+{\cal L})^{-1}= \dfrac{1}{14\cdot 15}\,
\left( \begin {array}{ccccccc} 210&0&0&0&0&0&0\\105&105&0&0&0&0&0\\ 0&0&120&30&60&0&0\\
0&0&60&120&30&0&0\\ 0&0&30&60&120&0&0\\56&0&8&2&4&112&28 \\14&0&32&8&16&28&112\\ \end {array} \right)\;.
\esenn
From this the pagerank vector follows from Theorem \ref{thm:pagerank}:
\bsenn
\wp = \dfrac{1}{294}
\left( \begin {array}{ccccccc} 77, &21, &50, &44, &46, &28, &28\end {array}\right)\;.
\esenn
The same calculation using ${\cal L}_t=I-S_t$ gives:
\bsenn
\wp_t = \dfrac{1}{273}
\left( \begin {array}{ccccccc} 56, &21, &50, &44, &46, &28, &28\end {array}\right)\;.
\esenn
Notice that the teleporting decreases the rank of vertex 1
and re-distributes the remaining measure over the other vertices.
It is easy to check that $L=\{1\}$ and $\wp_t(1)=\dfrac{(1-\beta)\wp(1)}{1-\beta \wp(1)}$
(Corollary \ref{cor:pagerank2}).

\begin{centering}
\section{Appendix: Discrete versus Continuous Consensus}
\label{chap:appendix}
\end{centering}
\setcounter{figure}{0} \setcounter{equation}{0}

\noindent We briefly discuss the way discrete and continuous consensus relate to one another.
The case of diffusion is exactly the same.

If we start with the flow defined by $\dot x=-{\cal L}x$ on a digraph $G$ then, we can easily
determine its time 1 map, namely $x^{(n+1)}=e^{-{\cal L}}x^{(n)}$, as well as its properties.
A graph $\tilde G$ is a \emph{transitive closure} of the graph $\tilde G$ if for every possible
(directed) path in $i\rightsquigarrow j$ in $G$ there is a (directed) edge $i \rightarrow j$ in
${\tilde G}$.

\begin{proposition} Consider a non-negative row-stochastic adjacency matrix $S$ associated
with the graph $G$ and a (rw) Laplacian
${\cal L}=I-S$. Then \\
(i) $e^{-{\cal L}}$ is row-stochastic,\\
(ii) $e^{-{\cal L}}$ is non-negative, and\\
(iii) $e^{-{\cal L}}$ is the adjacency matrix of a transitive closure of $G$.
\label{prop:app1}
\end{proposition}

\begin{proof} (i) From the expansion $e^{-{\cal L}}$ in powers of ${\cal L}$, we see that
$e^{-{\cal L}}$ must have row sum 1. (ii) If we expand $e^{-{\cal L}}=e^{S-I}=e^{-1}e^S$ in powers of
$S$, we see that $e^{-{\cal L}}$ is non-negative. (iii) By (i) and (ii), we may consider
$e^{-{\cal L}}$ as a (weighted) adjacency matrix ${\tilde S}$ of a graph ${\tilde G}$.
The $ji$ entry of $S^k$ positive if and only if there is a path $i \rightsquigarrow j$ of length $k$
in $G$. The expansion in $e^S$ shows that every power of S occurs with a positive coefficient.
Thus the $ji$ entry of ${\tilde S}$ is positive if and only if there is a path
$i \rightsquigarrow j$.
\end{proof}

Now let us start with a time 1 map $x^{(n+1)}=Sx^{(n)}$ and consider the much more challenging
problem of constructing (if possible) a flow $\dot x=-{\cal L}x$ that generates it. The matrix
${\cal L}$ is sometimes referred as the \emph{logarithm} of $S$ \cite{Hig}.
A stochastic matrix (or its corresponding Markov chain)
for which such a flow exists is called \emph{embeddable}. The problem of characterizing the
embeddable stochastic matrices is known as the \emph{embedding problem for finite Markov chains}
and is an area of current research interest.  The general problem of determining if a given stochastic
matrix has such a flow has been recently shown to be NP-complete \cite{cubitt}.  Recent characterizations for special classes of stochastic matrices can be found in \cite{jia} and \cite{vanbrunt}.

There is an obvious obstruction to the construction of a continuous flow $\dot x= -{\cal L}x$
that generates a given discrete system $x^{(n+1)}=Sx^{(n)}$ as its time 1 map
$x^{(n+1)}=e^{-{\cal L}}x^{(n)}$. Namely, if $S$ has an eigenvalue 0, then there is there is no
flow whose time 1 map generates it. Intuitively, this is because $e^{-{\cal L}}v=0$ would imply
that ${\cal L}v$ diverges, which, as we have seen, is impossible for a laplacian ${\cal L}$.
For more details we refer the reader to \cite{Hig} (under \emph{logarithms of matrices}).

A more interesting obstruction follows from Proposition \ref{prop:app1}. If the matrix $S$
does not correspond to a transitively closed graph, then the logarithm of $S$ (if it exists) cannot
be rw Laplacian. Here is an example.
\bsenn
S=\begin{pmatrix} 1 & 0 & 0 \\ 1/2 & 1/2 & 0\\ 0 & 3/5 & 2/5 \end{pmatrix}
\quad \logand \quad
\ln (S)=\begin{pmatrix} 0& 0 & 0 \\ \ln(2) & -\ln(2) & 0 \\
\ln(2^{11}/5^5) & \ln(5^6/2^{12}) & \ln(2/5) \end{pmatrix}\;.
\esenn

Perhaps surprisingly, transitive closure of the graph associated with a stochastic matrix $S$ is not sufficient to ensure that $S$ is embeddable.  Examples can be found in \cite{jia}. 
These examples give real logarithms of $S$ which are not Laplacians. One can show that there
are no other real logarithms (c.f. Theorem 2 in \cite{culver}). 

\begin{corollary}
The right kernels of ${\cal L}$ an ${\tilde {\cal L}}$ are equal. Similar for the left kernels.
\label{cory:app3}
\end{corollary}

\begin{proof} The structure of the reaches in a graph is entirely determined by the paths in that graph.  Thus, the reaches and in particular their subdivision in cabals, exclusive parts, and common parts are the same for $G$ and ${\tilde G}$. Thus dimension of the kernels is the same.
Since kernels are linear spaces, it is sufficient to show that
$\ker {\cal L}\subseteq \ker {\tilde {\cal L}}$. So, let $v\in {\cal L}$, then
\bsenn
{\tilde {\cal L}}v=\left(I-e^{-{\cal L}}\right)v=0\;,
\esenn
where the last equality follows upon expanding the exponential.  \end{proof}

{\bf Remark:} As we observed above, there is one substantial difference between $S$ an $e^{-\cal L}$.
While the strongly connected components of the two are equal, the restriction of ${\tilde S}$
to such a component is strictly positive. That means that the matrix restricted to such a component
is \emph{primitive}, and no periodic behavior occurs. Hence the averaging we see in Theorem
\ref{thm:equil-random-walker} for a discrete system is replaced by a straightforward limit
for continuous systems (Theorem \ref{thm:consensus}).

\begin{acknowledgements}\label{ackref}
The authors gratefully acknowledge a useful conversation
with Jeff Ovall.
\end{acknowledgements}

\affiliationone{
   J. J. P. Veerman and E. Kummel\\
   Portland State University\\
      PO Box 751\\
      Portland, OR 97207-0751\\
      United States of America
   \email{veerman@pdx.edu\\
   ewan@pdx.edu}}
%
\end{document}